\documentclass[12pt]{article}
\usepackage{amsmath,amssymb,amscd,latexsym,amsthm}
\usepackage{cases}

\newtheorem{Theorem}{Theorem}[section]

\newtheorem{Lemma}[Theorem]{Lemma}
\newtheorem{Corollary}[Theorem]{Corollary}

\newtheorem{Remark}[Theorem]{Remark}

\newcommand{\beq}{\begin{equation}}
\newcommand{\eeq}{\end{equation}}
\newcommand{\ben}{\begin{eqnarray}}
\newcommand{\een}{\end{eqnarray}}
\newcommand{\beno}{\begin{eqnarray*}}
\newcommand{\eeno}{\end{eqnarray*}}

\begin{document}

\title{Classification of  solutions of  the 2D steady Navier-Stokes equations with separated variables in cone-like domains}

\author{Wendong Wang\thanks{\small School of Mathematical Sciences, Dalian University of Technology, Dalian 116024, China; E-mail: wendong@dlut.edu.cn} \,
and\, Jie Wu\thanks{\small Center for Applied Mathematics, Tianjin University, Tianjin 300072, China; E-mail: jackwu@amss.ac.cn} }

\date{\today}

\maketitle

\begin{abstract}
We investigate the problem of classification of solutions for the steady Navier-Stokes equations in any cone-like domains.
In the form of separated variables,
$$u(x,y)=\left(
                              \begin{array}{c}
                                \varphi_1(r)v_1(\theta) \\
                                \varphi_2(r)v_2(\theta) \\
                              \end{array}
                            \right)
,$$
where $x=r\cos\theta$ and  $y=r\sin\theta$ in polar coordinates,
we obtain the expressions of all smooth solutions with $C^0$ Dirichlet boundary condition.  In particular, it shows that (i) some solutions are found, which are H\"{o}lder continuous on  the boundary, but their gradients blow up at the corner; (ii) all solutions in  the
entire plane of $\mathbb{R}^2$ like harmonic functions or Stokes equations, are polynomial expressions.
\end{abstract}




\section{Introduction}
Consider the incompressible Navier-Stokes equations on the whole space $\mathbb{R}^n$:
\begin{eqnarray}\label{SNS}
\left\{
  \begin{array}{ll}
    -\triangle u+u\cdot\nabla u+\nabla p=0,  \\
   \mbox{div}~u=0.
  \end{array}
\right.
\end{eqnarray}
One challenging problem is
to investigate the classification of solutions of  (\ref{SNS}), which is related to the celebrated paper by Koch-Nadirashvili-Seregin-Sverak \cite{KNSS}, where the regularity problem of the non-stationary Navier-Stokes equations was reduced to Liouville type theorem of bounded ancient solutions. As they said in \cite{KNSS} ``The case of general 3-dimensional fields is,
as far as we know, completely open. In fact, it is open even in the steady-state case ($u$
independent of $t$)."
By assuming the additional finite Dirichlet integral, Liouville theorem is proved by Galdi \cite{Galdi} for $u\in L^{\frac92}(\mathbb{R}^3)$. For more references, we refer to \cite{Chae}, \cite{Se} and the references therein.

However, even for the 2-dimensional case, the problem of classification of solutions for the steady Navier-Stokes equations is not solved.
As the case of 2D, Gilbarg-Weinberger  \cite{GW1978} proved the above Liouville type theorem by assuming the finite Dirichlet integral, where they made use of the fact that the vorticity function satisfies
a nice elliptic equation to which a maximum principle applies. When the smooth solution $u$ is bounded, a Liouville theorem being more in the spirit of the classical one for entire analytic functions
was obtained by in \cite{KNSS} as a byproduct of their
work on the non-stationary case. If $\nabla u\in L^{q}(\mathbb{R}^2)$ with $1<p<\infty$, the constant $u$ follows by the first author in \cite{Wang} by using the growth estimate of $\mathcal{D}^{1,q}$ functions (see also \cite{KTW} for another approach). As suggested by Fuchs-Zhong in \cite{FZ2011}:
\begin{center}
``\emph{Suppose that
$\lim_{|x|\rightarrow\infty}|x|^{-1}|u(x)|=0.$ Does the constancy of u follow?}"
\end{center}
It's true for harmonic functions,
 since the linear solutions are the counterexamples; see also Yau \cite{Yau} and Li-Tam \cite{Li-Tam}, where they considered the space of harmonic functions on complete manifold with
nonnegative Ricci curvature with linear growth. When $
\limsup_{|x|\rightarrow\infty}|x|^{-\alpha}|u(x)|\leq C
$ with $\alpha\in [0, 1/7)$, $u$ is a constant vector  \cite{FZ2011}.
The component is improved to $\alpha<\frac13$ with help of the vorticity equation  by Bildhauer-Fuchs-Zhang in \cite{BFZ2013}.

In this note, our purpose is to classify the solutions for the 2D steady Navier-Stokes equations by separating variables.

Let $\Omega$ be the whole space $\mathbb{R}^2$, the half-space $\mathbb{R}^2_+$, or any cone domain of
$\{(r,\theta); \alpha<\theta<\beta,~0<r<\infty\}$ with $0\leq \alpha<\beta\leq2\pi$.
Our first result is as follows.
\begin{Theorem}\label{mainthm1}
Suppose that
$(u, p)\in C^3(\Omega)\times C^1(\Omega)$ is a solution of (\ref{SNS}), and $u\in C^0(\bar{\Omega})$, which has the form
$$u(x,y)=\varphi(r)\left(
                              \begin{array}{c}
                                v_1(\theta) \\
                                v_2(\theta) \\
                              \end{array}
                            \right).$$
Then, $(u, p)$ can only be expressed in one of the following types:

(i)\begin{equation*}\label{intro1}
  u=\left(
       \begin{array}{c}
         C_1 \\
         C_2\\
       \end{array}
     \right),~~~
     p=C_3;
\end{equation*}

(ii)\begin{equation*}\label{intro2}
  u=\left(
           \begin{array}{c}
           C_1x+C_2y \\
           C_3x-C_1y \\
           \end{array}
           \right),~~~
 p=-\frac{1}{2}(C_1^2+C_2C_3)(x^2+y^2)+C_4;
\end{equation*}

(iii)\begin{equation*}\label{intro3}
\begin{split}
&u=\left(
       \begin{array}{c}
         (C_2+C_3)x^2+(3C_2-C_3)y^2+2(C_1+C_4)xy\\
         (C_4-3C_1)x^2-(C_1+C_4)y^2-2(C_2+C_3)xy\\
       \end{array}
     \right),\\
&p=\frac{1}{2}(C_1^2+C_2^2-C_3^2-C_4^2)(x^2+y^2)^2+8C_2x-8C_1y+C_5,
\end{split}
\end{equation*}
with $C_1, C_2, C_3, C_4$ satisfying
\begin{equation}\label{intro3a}
\left\{
  \begin{array}{ll}
    C_1C_3+C_2C_4-2C_1C_2=0,  \\
   C_1C_4-C_2C_3+C_1^2-C_2^2=0.
  \end{array}
\right.
\end{equation}

(iv)\begin{equation*}\label{intro4}
  u=r^\lambda\left(
       \begin{array}{c}
         C_1\cos(\lambda\theta)+C_2\sin(\lambda\theta) \\
         C_2\cos(\lambda\theta)-C_1\sin(\lambda\theta)\\
       \end{array}
     \right),~~~
 p=-\frac{1}{2}(C_1^2+C_2^2)r^{2\lambda}+C_3.
\end{equation*}
If $\Omega=\mathbb{R}^2$, $\lambda\geq3$ and $\lambda\in\mathbb{N}$,
otherwise $\lambda\in (0,1)\cup(1,2)\cup(2,\infty)$.

(v) If $\Omega\neq\mathbb{R}^2$,
\beno\label{case1}
u=(C_1+C_2\ln r)\left(
                              \begin{array}{c}
                                -y\\
                                x \\
                              \end{array}
                            \right),
\eeno
\beno
p=\frac12 r^2\Big[C_2^2\ln^2r+(2C_1C_2-C_2^2)\ln r+C_1^2-C_1C_2+\frac{1}{2}C_2^2\Big]+2C_2\theta+C_3.
\eeno
\end{Theorem}

\begin{Remark}\label{remark1}
In the condition of (iii), the coefficients of $u_1$ are proportional to those of $u_2$, since equations (\ref{intro3a}) are equivalent to
$$\frac{C_2+C_3}{C_4-3C_1}=\frac{3C_2-C_3}{-(C_1+C_4)}=\frac{(C_1+C_4)}{-(C_2+C_3)}.$$
\end{Remark}

\begin{Remark}[Boundary blow-up phenomenon] First,
the solutions in (v) show that
\beno
  \nabla u(x,y)=\left(
                              \begin{array}{cc}
                               -C_2\sin\theta\cos\theta  &-(C_1+C_2\ln r+C_2\sin^2\theta) \\
                               C_1+C_2\ln r+C_2\cos^2\theta  & C_2\sin\theta\cos\theta\\
                              \end{array}
                            \right),
\eeno
which blow up at the corner of $r=0$. However, $u\in C^\gamma(\bar{\Omega})$ locally for any $0<\gamma<1.$ This is different from the case in \cite{KLLT}, where the authors consider a class of H\"{o}lder continuous boundary data on the time and prove there exist unbounded gradients at boundary. Second, our examples show that the singularity of the solution does not depend on the regularity of the boundary (for example, the case of $\mathbb{R}^2_+$ when $\alpha=0,\beta=\pi$). Also, for boundary regularity criteria of steady Navier-Stokes equations (for example, see \cite{Liu-Wang} for the boundary  H\"{o}lder regularity of 6D steady Navier-Stokes equations), it's impossible to prove the uniform boundary $C^\gamma$ regularity with $\gamma>\gamma_0>0$ independent of $u$ due to the examples in (iv) for nontrivial boundary data.
\end{Remark}


More generally, let $u$ be the form of
\begin{equation}\label{gene.form}
  u(x,y)=\left(
                              \begin{array}{c}
                                v_1(\theta)\varphi_1(r) \\
                                v_2(\theta)\varphi_2(r) \\
                              \end{array}
                            \right),
\end{equation}
and we have the following conclusions.
\begin{Theorem}\label{mainthm3}
Suppose that
$(u, p)\in C^3(\Omega)\times C^1(\Omega)$ is a solution of (\ref{SNS}) with the form of (\ref{gene.form}). Moreover, $u\in C^0(\bar{\Omega})$.
Then, $(u, p)$  can only be expressed as one of the forms in
(i), (ii), (iii), (iv) and (v)  in Theorem \ref{mainthm1}, or one of the following two types:
\begin{equation*}\label{2ndcase2.2.2.1}
u=\left(
       \begin{array}{c}
         C_1 \\
         C_2x\\
       \end{array}
     \right),~~~
     p=-C_1C_2y+C_3,
\end{equation*}
and
\begin{equation*}
u=\left(
       \begin{array}{c}
         C_1y \\
         C_2\\
       \end{array}
     \right),~~~
     p=-C_1C_2x+C_3.
\end{equation*}
\end{Theorem}
This theorem immediately leads to the following conclusion.
\begin{Corollary}\label{coro1}
Suppose that
$(u, p)$ satisfies the assumptions of Theorem \ref{mainthm3} with $\Omega=\mathbb{R}^2$, then $u$ and $p$ must be polynomials, which is similar to harmonic functions on the whole space.
\end{Corollary}

As an application, we prove a sharp and extended Liouville theorem to any cone domains for (\ref{SNS}) when $u$ has the form (\ref{gene.form}), which  answer the question in \cite{FZ2011} in this setting.
\begin{Corollary}\label{coro2}
Suppose that
$(u, p)$ satisfies the assumptions of Theorem \ref{mainthm3} with $u\in C^1(\bar{\Omega})$ and
\begin{equation}\label{growth}
\lim_{|x|\rightarrow\infty}|x|^{-1}|u(x)|=0,
\end{equation}
then $(u, p)$ must be constant.
\end{Corollary}

\section{Preliminaries}

We state some preliminary lemmas before proving the main theorems, which play important roles in our following arguments.

In this part, we let $I$ and $J$ be intervals in $\mathbb{R}$.

\begin{Lemma}\label{lemma1}
Suppose that
\begin{equation}\label{lemma1a}
  A(\theta)f(r)=B(\theta)g(r), ~~~~~\theta\in I, ~~~r\in J,
\end{equation}
If $g(r)\not\equiv 0$, then either
$$A(\theta)=B(\theta)\equiv 0,$$
or, there exists a constant $\lambda$ such that
$$B(\theta)=\lambda A(\theta), ~~~~~f(r)=\lambda g(r).$$
\end{Lemma}

\begin{proof}
Assume that $g(r_0)\neq0$ without loss of generality.

\textbf{Case 1}: $A(\theta)\equiv0$. Then
$B(\theta)g(r)\equiv0$ for $\theta\in I$ and $r\in J,$
which implies
$B(\theta)\equiv0$ for $\theta\in I$ due to $g(r_0)\neq0$.

\textbf{Case 2}: $A(\theta)\not\equiv0$. We assume that $A(\theta_0)\neq0$ for some $\theta_0\in I$, then
\begin{equation}\label{lemma1b}
  f(r)=\frac{B(\theta_0)}{A(\theta_0)}g(r)=:\lambda g(r),~~~~~r\in J,
\end{equation}
where $\lambda=\frac{B(\theta_0)}{A(\theta_0)}$. Substituting equation (\ref{lemma1b}) into (\ref{lemma1a}), there holds
$$  \lambda A(\theta)g(r)=B(\theta)g(r), ~~~~~\theta\in I, ~~~r\in J,$$
which yields
$B(\theta)=\lambda A(\theta)$ for $\theta\in I$ by taking $r=r_0.$
The proof is complete.
\end{proof}

\begin{Lemma}\label{lemma2new}Let $v=v(\theta), ~\theta\in I$.\\
(1) If $v\in C^1(I)$ and satisfies
$\sin\theta v+\cos\theta v'=0,$
then
$v=C\cos\theta.$\\
(2) If $v\in C^2(I)$ and satisfies
$2\sin\theta v+\cos\theta v'=0,$
then
$v=C\cos^2\theta.$
\end{Lemma}

\begin{proof}
 Assume that $I=(0,2\pi)$ for simplicity and denote
$$I_1=(0, \frac{\pi}{2}),~~~I_2=(\frac{\pi}{2}, \frac{3\pi}{2}),~~~~I_3=(\frac{3\pi}{2}, 2\pi),$$
then
$$\cos\theta\neq0,~~~~~\theta\in I_i,~~i=1, 2, 3.$$

(1) In $I_i, ~i=1, 2, 3,$
$$\big(\cos^{-1}\theta v\big)'=\cos^{-2}\theta(\sin\theta v+\cos\theta v')=0,$$
then $$\cos^{-1}\theta v=C_i,~~~v=C_i\cos\theta,~~~~~\theta\in I_i,$$
and
$$v'=-C_i\sin\theta,~~~~~~\theta\in I_i.$$
Then $C_1=C_2=C_3=:C,$ since $v'$ is continuous at $\frac{\pi}{2}$ and $\frac{3\pi}{2}$.
Therefore
$v=C\cos\theta.$

(2) The argument is simialr, and we omitted it.


The proof is complete.
\end{proof}



\begin{Lemma}\label{lemma3}
Suppose that $\varphi_1,~\varphi_2\in C^1 \big((0,+\infty)\big)$ and satisfy
\begin{equation}\label{lem3eqn}
\left\{
  \begin{array}{ll}
r\varphi_1'(r)=a\varphi_1+b\varphi_2,  \\
r\varphi_2'(r)=c\varphi_1+d\varphi_2.
  \end{array}
\right.
\end{equation}
Moreover, let $\delta:=(a-d)^2+4bc$. Then\\
(1) If $b=0$, $d=a$, then
\begin{equation*}
\left\{
    \begin{array}{ll}
\varphi_1=C_1r^a, \\
\varphi_2=(cC_1\ln r+C_2)r^a.
    \end{array}
  \right.
\end{equation*}
(2) If $b=0$, $d\neq a$, then
\begin{equation*}
\left\{
    \begin{array}{ll}
\varphi_1=C_1r^a, \\
\varphi_2=\frac{c}{a-d}C_1r^a+C_2r^d.
    \end{array}
  \right.
\end{equation*}
(3) If $b\neq0$, $\delta>0$, then
\begin{equation*}
\left\{
    \begin{array}{ll}
\varphi_1=C_1r^m+C_2r^n, ~~~~~~~(m>n)\\
\varphi_2=\frac{m-a}{b}C_1r^m+\frac{n-a}{b}C_2r^n,
    \end{array}
  \right.
\end{equation*}
where $m, n$ are two different real roots of equation
\begin{equation}\label{char.eqn}
  \rho^2-(a+d)\rho+ad-bc=0.
\end{equation}
(4) If $b\neq0$, $\delta=0$, then
\begin{equation*}
\left\{
    \begin{array}{ll}
\varphi_1=(C_1\ln r+C_2)r^l, \\
\varphi_2=\big[\frac{l-a}{b}C_1\ln r+\frac{C_1+(l-a)C_2}{b}\big]r^l,
    \end{array}
  \right.
\end{equation*}
where $l$ is the unique real root of  (\ref{char.eqn}).\\
(5) If $b\neq0$, $\delta<0$, then
\begin{equation*}
\left\{
    \begin{array}{ll}
\varphi_1=\big[C_1\cos(\mu\ln r)+C_2\sin(\mu\ln r)\big]r^\lambda, \\
\varphi_2=\big[\frac{(\lambda-a) C_1+\mu C_2}{b}\cos(\mu \ln r)+\frac{(\lambda-a) C_2-\mu C_1}{b}\sin(\mu\ln r)\big]r^\lambda,
    \end{array}
  \right.
\end{equation*}
where $\lambda\pm \mu i$ are the complex roots of (\ref{char.eqn}).
\end{Lemma}

\begin{proof}
Let $r=e^t$ and $D=\frac{d}{dt}$, then the equations (\ref{lem3eqn}) become
\begin{equation}\label{lem3eqn-1}
\left\{
  \begin{array}{ll}
D\varphi_1=a\varphi_1+b\varphi_2,  \\
D\varphi_2=c\varphi_1+d\varphi_2.
  \end{array}
\right.
\end{equation}

\textbf{\underline{Case 1: $b=0$.}} The first equation of (\ref{lem3eqn-1}) becomes
$D\varphi_1=a\varphi_1,$
then
\begin{equation}\label{lem3eqn-2}
 \varphi_1=C_1e^{at}=C_1r^a.
\end{equation}
Substituting (\ref{lem3eqn-2}) into the second equation of (\ref{lem3eqn-1}), we get
$$D\varphi_2-d\varphi_2=cC_1e^{at}.$$
Then
\begin{equation}\label{lem3.case1}
  D\big(e^{-dt}\varphi_2\big)=e^{-dt}(D\varphi_2-d\varphi_2)=cC_1e^{(a-d)t},
\end{equation}
which can be divided into the following two situations.

\textbf{Case 1.1}: If $d=a$,
there holds $e^{-at}\varphi_2=cC_1t+C_2,$ and
$$\varphi_2=(cC_1t+C_2)e^{at}=(cC_1\ln r+C_2)r^a.$$

\textbf{Case 1.2}: If $d\neq a$, it follows from (\ref{lem3.case1}) that
$e^{-dt}\varphi_2=\frac{c}{a-d}C_1e^{(a-d)t}+C_2,$ and
$$\varphi_2=\frac{c}{a-d}C_1e^{at}+C_2e^{dt}=\frac{c}{a-d}C_1r^a+C_2r^d.$$

\textbf{\underline{Case 2: $b\neq0$.}} The first equation of (\ref{lem3eqn-1}) implies that
\begin{equation}\label{lem3eqn-3}
\varphi_2=\frac{1}{b}(D\varphi_1-a\varphi_1).
\end{equation}
Substituting (\ref{lem3eqn-3}) into  $(\ref{lem3eqn-1})_2$, we get
\begin{equation}\label{lem3eqn-4}
  D^2\varphi_1-(a+d)D\varphi_1+(ad-bc)\varphi_1=0,
\end{equation}
which has characteristic equation (\ref{char.eqn}).

\textbf{Case 2.1}: If $\delta>0$, equation (\ref{char.eqn}) has two different real roots $m, n~(m>n)$
and  the general solution of (\ref{lem3eqn-4}) is expressed as follows
$$
\varphi_1=C_1e^{mt}+C_2e^{nt}=C_1r^m+C_2r^n.
$$
Substituting this into (\ref{lem3eqn-3}), we get
$$
\varphi_2=\frac{m-a}{b}C_1e^{mt}+\frac{n-a}{b}C_2e^{nt}=\frac{m-a}{b}C_1r^m+\frac{n-a}{b}C_2r^n.
$$

\textbf{Case 2.2}: If $\delta=0$, equation (\ref{char.eqn}) has a unique real root $l$
and the general solution of (\ref{lem3eqn-4}) is
$$\varphi_1=(C_1t+C_2)e^{lt}=(C_1\ln r+C_2)r^l.$$
Substituting this into (\ref{lem3eqn-3}), we get
$$
\varphi_2
=\Big[\frac{l-a}{b}C_1\ln r+\frac{C_1+(l-a)C_2}{b}\Big]r^l.
$$

\textbf{Case 2.3}: If $\delta<0$, equation (\ref{char.eqn}) has complex roots $\lambda\pm\mu i~(\mu\neq0)$
and the general solution of (\ref{lem3eqn-4}) is
$$
\varphi_1=\big[C_1\cos(\mu t)+C_2\sin(\mu t)\big]e^{\lambda t}=\big[C_1\cos(\mu\ln r)+C_2\sin(\mu\ln r)\big]r^\lambda.
$$
Substituting this into (\ref{lem3eqn-3}), we get
\begin{equation*}
  \begin{split}
  \varphi_2=\Big[\frac{(\lambda-a) C_1+\mu C_2}{b}\cos(\mu \ln r)+\frac{(\lambda-a) C_2-\mu C_1}{b}\sin(\mu\ln r)\Big]r^\lambda.
  \end{split}
\end{equation*}
Thus  the proof is complete.
\end{proof}

\begin{Lemma}\label{lemma4}
Suppose that $v_1,~v_2\in C^1(I)$ satisfying
\begin{equation}\label{lem4eqn-1}
\left\{
  \begin{array}{ll}
a\cos\theta v_1+c\sin\theta v_2-\sin\theta v_1'=0, \\
b\cos\theta v_1+d\sin\theta v_2+\cos\theta v_2'=0,
  \end{array}
\right.
\end{equation}
and $v_1v_2\equiv0$. Then we have (i) $v_1\equiv0$ if $b\neq0$; (ii) $v_2\equiv0$ if $c\neq0$.
\end{Lemma}

\begin{proof}
Multiplying $(\ref{lem4eqn-1})_1$ by $v_2$, due to $v_1v_2\equiv0$  we get
\begin{equation}\label{lem4eqn-2}
cv_2^2-v_1'v_2=0.
\end{equation}
Similarly, multiplying $(\ref{lem4eqn-1})_2$ by $v_1$, it follows that
\begin{equation}\label{lem4eqn-3}
bv_1^2+v_1v_2'=0.
\end{equation}
Then (\ref{lem4eqn-3}) minus (\ref{lem4eqn-2}) tells us
$bv_1^2-cv_2^2+(v_1v_2)'=0.$
Then $bv_1^2=cv_2^2,$ since $v_1v_2\equiv0$.
Consequently,
$$bv_1^3=cv_2(v_1v_2)\equiv0,\quad cv_2^3=bv_1(v_1v_2)\equiv0.$$
If $b\neq0$, then $v_1\equiv0.$ If $c\neq0$, then $v_2\equiv0.$
The proof is complete.
\end{proof}

\section{Proof of Theorem \ref{mainthm1}}
\begin{proof}[Proof of Theorem \ref{mainthm1}]
Let $w:=\partial_2u_1-\partial_1u_2$ be the vorticity of $u$, then $w$ satisfies the equation
\begin{equation}\label{weqn}
  \triangle w-u\cdot\nabla w=0.
\end{equation}
Throughout this section, we write $v_i(\theta)$, $v_i'(\theta)$, $\varphi(r)$, $\varphi'(r)$ as $v_i$, $v_i'$, $\varphi$, $\varphi'$, $i=1,2.$

Direct computations show that
\begin{equation}\label{divu}
\begin{split}
 \mbox{div}~u=&\big(\cos\theta v_1+\sin\theta v_2\big)\varphi'-\big(\sin\theta v_1'-\cos\theta v_2'\big)\frac{\varphi}{r}\\
     =&:A(\theta)\varphi'-B(\theta)\frac{\varphi}{r};
 \end{split}
\end{equation}
\begin{equation}\label{wexpr}
  w=\big(\sin\theta v_1-\cos\theta v_2\big)\varphi'+\big(\cos\theta v_1'+\sin\theta v_2'\big)\frac{\varphi}{r}.
\end{equation}
The equation $\mbox{div}~u=0$ and (\ref{divu}) yield that
$A(\theta)\varphi'=B(\theta)\frac{\varphi}{r}.$
By Lemma \ref{lemma1}, we have either
$$A(\theta)=B(\theta)\equiv 0,$$
or
$$B(\theta)=\lambda A(\theta), ~~~~~\varphi'=\lambda\frac{\varphi}{r}.$$
Next, we discuss two cases respectively.

\textbf{\underline{Step 1: $A(\theta)=B(\theta)\equiv 0$.}}
$A(\theta)\equiv0$ implies that
\begin{equation}\label{kind1case1a}
v_1=-\tan \theta v_2,
\end{equation}
and $B(\theta)\equiv0$ tells us
\begin{equation}\label{kind1case1b}
-\tan \theta v_1'+v_2'=0.
\end{equation}
Substituting (\ref{kind1case1a}) into (\ref{kind1case1b}), we deduce that
$\sin\theta v_2+\cos\theta v_2'=0.$
Due to $v_2\in C^1$, applying Lemma \ref{lemma2new} we obtain that
\begin{equation}\label{case1-v2expr}
  v_2=C\cos\theta,
\end{equation}
and thus
\begin{equation}\label{case1-v1expr}
 v_1=-C\sin\theta.
\end{equation}
Without loss of generality, we assume that $C\neq0$.
Substituting (\ref{case1-v2expr}) and (\ref{case1-v1expr}) into (\ref{wexpr}), we have
$$w=-C\big(\varphi'+\frac{\varphi}{r}\big),$$
and then
\begin{equation}\label{w1}
\begin{split}
 \Delta w=&\big(\partial_r^2+\frac{1}{r}\partial_r+\frac{1}{r^2}\partial_\theta^2\big)w=-C\big(\partial_r^2+\frac{1}{r}\partial_r\big)\big(\varphi'+\frac{\varphi}{r}\big)\\
  =&-C(\varphi'''+\frac{2\varphi''}{r}-\frac{\varphi'}{r^2}+\frac{\varphi}{r^3}\big);\\
  \partial_1w=&-C\big(\varphi'+\frac{\varphi}{r}\big)'\cos \theta=-C\cos\theta\big(\varphi''+\frac{\varphi'}{r}-\frac{\varphi}{r^2}\big);\\
  \partial_2w=&-C\sin\theta\big(\varphi''+\frac{\varphi'}{r}-\frac{\varphi}{r^2}\big);\\
  u\cdot \nabla w=&\varphi(v_1\partial_1w+v_2\partial_2w)=0.
\end{split}
\end{equation}
Combining (\ref{weqn}) and (\ref{w1}), we obtain
\begin{equation}\label{phieqn}
r^3\varphi'''+2r^2\varphi''-r\varphi'+\varphi=0.
\end{equation}
Let $r=e^t$ and denote $D=\frac{d}{dt}$, then
the equation (\ref{phieqn}) becomes
$$(D+1)(D-1)^2\varphi=0,$$
which has the general solution as
\begin{equation*}
     \varphi=C_1e^t+C_2te^t+C_3e^{-t}=C_1r+C_2r\ln r+C_3r^{-1}.
\end{equation*}
Recall that $u\in C^3(\Omega)\cap C^0(\bar{\Omega})$, then
\begin{equation}\label{wb}
\varphi=\left\{
  \begin{array}{ll}
    C_1r, & \hbox{if $\Omega=\mathbb{R}^2$;} \\
    C_1r+C_2r\ln r, & \hbox{if $\Omega\neq\mathbb{R}^2$.}
  \end{array}
\right.
\end{equation}
By (\ref{case1-v2expr}), (\ref{case1-v1expr}) and (\ref{wb}),
\begin{equation}\label{case1}
u=\left\{
  \begin{array}{ll}
   C_1'\left(
                              \begin{array}{c}
                                -y \\
                                x \\
                              \end{array}
                            \right), & \hbox{if $\Omega=\mathbb{R}^2$;} \\
   (C_1'+C_2'\ln r)\left(
                              \begin{array}{c}
                                -y \\
                                x \\
                              \end{array}
                            \right), & \hbox{if $\Omega\neq\mathbb{R}^2$.}
  \end{array}
\right.
\end{equation}

\textbf{\underline{Step 2: $B(\theta)=\lambda A(\theta), ~\varphi'=\lambda\frac{\varphi}{r}$.}} At this time, we have
\begin{equation}\label{phiexpr}
  \varphi=C r^\lambda.
\end{equation}
Without loss of generality, we assume that $C\neq0$. Denote
$$L(\theta)=\sin\theta v_1-\cos\theta v_2,$$
then it is easy to verify that
\begin{equation}\label{La}
  \begin{split}
  L'=A+B=(\lambda+1)A,\quad
\cos\theta v_1'+\sin\theta v_2'=A'+L.
   \end{split}
\end{equation}
By (\ref{wexpr}), (\ref{phiexpr}) and (\ref{La}) we have
$$w=C\lambda r^{\lambda-1}L+Cr^{\lambda-1}(A'+L)=Cr^{\lambda-1}\big[A'+(\lambda+1)L\big]=:Cr^{\lambda-1}H,$$
where
\begin{equation}\label{HAL}
  H=A'+(\lambda+1)L.
\end{equation}
Then we have
\begin{equation}\label{wc}
\begin{split}
 \Delta w=&\big(\partial_r^2+\frac{1}{r}\partial_r+\frac{1}{r^2}\partial_\theta^2\big)w \\
    =&CH\big(\partial_r^2+\frac{1}{r}\partial_r\big)r^{\lambda-1}+Cr^{\lambda-3}H''\\
    =&Cr^{\lambda-3}\big[H''+(\lambda-1)^2H\big];\\
 u\cdot \nabla w=&\varphi(v_1\partial_1 w+v_2\partial_2 w)\\
 =&\varphi\Big[v_1\Big(\partial_r w\cos\theta-\partial_\theta w\frac{\sin\theta}{r}\Big)+v_2\Big(\partial_r w\sin\theta+\partial_\theta w\frac{\cos\theta}{r}\Big)\Big]\\
=&\varphi\big(\partial_r wA-\frac{\partial_\theta w}{r}L\big) \\
=&C^2r^{2\lambda-2}\big[(\lambda-1)HA-H'L\big].
\end{split}
\end{equation}
Combing (\ref{weqn}) and (\ref{wc}), we have
 \begin{equation*}\label{eqndis}
  Cr^{\lambda+1}\big[(\lambda-1)HA-H'L\big]=H''+(\lambda-1)^2H.
 \end{equation*}
Since $u$ is differentiable, then $\lambda\geq 0$ and the above equation is equivalent to
\begin{numcases}{}
 (\lambda-1)HA-H'L=0, \label{HLeqna} \\
H''+(\lambda-1)^2H=0. \label{HLeqn}
\end{numcases}
We keep in mind that
\begin{equation}\label{kim}
 L'=(\lambda+1)A, ~~~~H=A'+(\lambda+1)L.
\end{equation}
First, we can solve $H$ according to equation (\ref{HLeqn}). If $\lambda=1$, it is easy. If $\lambda\neq1$, (\ref{HLeqn}) has general solution
\begin{equation}\label{11}
  H=A'+(\lambda+1)L=C_1\cos \big((\lambda-1)\theta\big)+C_2\sin \big((\lambda-1)\theta\big),~~~~~\lambda\neq1.
\end{equation}
Substituting (\ref{11}) into equation (\ref{HLeqna}), we have
\begin{equation}\label{12}
  \Big[C_1\cos \big((\lambda-1)\theta\big)+C_2\sin \big((\lambda-1)\theta\big)\Big]A-
  \Big[C_2\cos \big((\lambda-1)\theta\big)-C_1\sin \big((\lambda-1)\theta\big)\Big]L=0.
\end{equation}
Differentiate both sides of (\ref{11}) and we obtain that
$$A''+(\lambda+1)L'=(\lambda-1)\Big[C_2\cos \big((\lambda-1)\theta\big)-C_1\sin \big((\lambda-1)\theta\big)\Big].$$
This and (\ref{kim}) yield that
\begin{equation*}\label{13}
  A''+(\lambda+1)^2A=(\lambda-1)\Big[C_2\cos \big((\lambda-1)\theta\big)-C_1\sin \big((\lambda-1)\theta\big)\Big],~~~~~\lambda\neq1.
\end{equation*}
This equation has general solution
\begin{equation}\label{15a}
\begin{split}
 A=& C_3\cos \theta+C_4\sin\theta+\frac{\theta}{2}(C_1\cos \theta-C_2\sin\theta)\\
 =& \cos \theta\Big(\frac{C_1\theta}{2}+C_3\Big)+\sin \theta\Big(-\frac{C_2\theta}{2}+C_4\Big),~~~~~~~~~~\lambda=0;\\
A=&C_3\cos\big((\lambda+1)\theta\big)+C_4\sin\big((\lambda+1)\theta\big)\\
&+\frac{\lambda-1}{4\lambda}\Big[C_2\cos \big((\lambda-1)\theta\big)-C_1\sin \big((\lambda-1)\theta\big)\Big],~~~~~\lambda\neq0, 1.
\end{split}
\end{equation}
Next, the classification of solutions is discussed in the following cases.

\textbf{\underline{Case 2.1: $\lambda=1$.}} Equations (\ref{HLeqna}) and (\ref{HLeqn}) become $H'L=0$ and
   $H''=0,$
then $H=a\theta+b$ and $aL=0$.
If $a\neq0$, then $L=0$. By (\ref{kim}), $A=\frac{1}{2}L'=0$, $H=A'+2L=0$, this contradicts that $a\neq0$. Therefore $a=0$ and $H=b$. Namely $A'+2L=b,$
then $A''+2L'=0.$
By (\ref{kim}), $L'=2A$, then
$A''+4A=0.$
This equation has general solution
\begin{equation}\label{Alast}
A=C_1\cos 2\theta+C_2\sin 2\theta.
\end{equation}
Then
\begin{equation}\label{Llast}
 L=\frac{1}{2}(H-A')=\frac{b}{2}+C_1\sin 2\theta-C_2\cos 2\theta.
\end{equation}
Combining (\ref{Alast}) and (\ref{Llast}), we obtain
\begin{eqnarray*}
\left\{
  \begin{array}{ll}
    v_1=C_1\cos\theta+\big(C_2+\frac{b}{2}\big)\sin\theta,  \\
   v_2=\big(C_2-\frac{b}{2}\big)\cos\theta-C_1\sin\theta.
  \end{array}
\right.
\end{eqnarray*}
Then
\begin{equation}\label{case2}
  u=Cr\left(
                              \begin{array}{c}
                                C_1\cos\theta+\big(C_2+\frac{b}{2}\big)\sin\theta \\
                                \big(C_2-\frac{b}{2}\big)\cos\theta-C_1\sin\theta \\
                              \end{array}
                            \right)
=\left(
                              \begin{array}{c}
                                C_1'x+C_2'y \\
                                C_3'x-C_1'y \\
                              \end{array}
                            \right).
\end{equation}

\textbf{\underline{Case 2.2: $\lambda=0$.}} 
As shown in (\ref{15a}), in this case
\begin{equation}\label{15a0}
 A=\cos \theta\Big(\frac{C_1\theta}{2}+C_3\Big)+\sin \theta\Big(-\frac{C_2\theta}{2}+C_4\Big).
\end{equation}
By (\ref{kim}), (\ref{11}) and (\ref{15a0}), we have
\begin{equation}\label{15b}
  L=H-A'=\cos \theta\Big(\frac{C_1+C_2\theta}{2}-C_4\Big)+\sin \theta\Big(\frac{C_1\theta-C_2}{2}+C_3\Big).
\end{equation}
We substitute (\ref{15a0}) and (\ref{15b}) into (\ref{12}) and obtain that
\begin{equation*}
\begin{split}
&\Big[\frac{(C_1^2-C_2^2)\theta-C_1C_2}{2}+C_1C_3+C_2C_4\Big]\cos2\theta \\
-&\Big[C_1C_2\theta-C_1C_4+C_2C_3+\frac{C_1^2-C_2^2}{4}\Big]\sin 2\theta=0.
\end{split}
\end{equation*}
Applying Lemma \ref{lemma1} again, we get
$$\frac{(C_1^2-C_2^2)\theta-C_1C_2}{2}+C_1C_3+C_2C_4=C_1C_2\theta-C_1C_4+C_2C_3+\frac{C_1^2-C_2^2}{4}\equiv0,$$
and thus
\begin{equation}\label{c1c20}
  C_1=C_2=0.
\end{equation}
By (\ref{15a0}), (\ref{15b}) and (\ref{c1c20}), we have
\begin{eqnarray*}
\left\{
  \begin{array}{ll}
    A=C_3\cos\theta+C_4\sin\theta \\
    L=C_3\sin\theta-C_4\cos\theta.
  \end{array}
\right.
\end{eqnarray*}
Then
\beno
    v_1=C_3,\quad
    v_2=C_4
  \eeno
and
\begin{equation}\label{case221}
  u=C\left(
       \begin{array}{c}
         C_3 \\
         C_4\\
       \end{array}
     \right)
     =\left(
       \begin{array}{c}
         C_1' \\
         C_2' \\
       \end{array}
     \right).
\end{equation}

Moreover, for $\lambda\neq0, 1$, as shown in (\ref{15a}) there holds
\begin{equation}\label{15ano01}
A=C_3\cos\big((\lambda+1)\theta\big)+C_4\sin\big((\lambda+1)\theta\big)\\
+\frac{\lambda-1}{4\lambda}\Big[C_2\cos \big((\lambda-1)\theta\big)-C_1\sin \big((\lambda-1)\theta\big)\Big],
\end{equation}
and due to  (\ref{kim}), (\ref{11}) and (\ref{15ano01}) we have
\begin{equation}\label{16b}
\begin{split}
 L=&\frac{1}{\lambda+1}(H-A')\\
 =&C_3\sin\big((\lambda+1)\theta\big)-C_4\cos\big((\lambda+1)\theta\big)\\
 &+\frac{\lambda+1}{4\lambda}\Big[C_1\cos\big((\lambda-1)\theta\big)+C_2\sin\big((\lambda-1)\theta\big)\Big].
\end{split}
\end{equation}
Furthermore, by substituting (\ref{15ano01}) and (\ref{16b}) into (\ref{12}) we obtain
\begin{equation}\label{case222}
  \begin{split}
    &(C_1C_3+C_2C_4)\cos2\theta+(C_1C_4-C_2C_3)\sin2\theta\\
-&\frac{1}{2\lambda}\Big[C_1C_2\cos\big((2\lambda-2)\theta\big)+\frac{C_2^2-C_1^2}{2}\sin\big((2\lambda-2)\theta\big)\Big]=0.
  \end{split}
\end{equation}

\textbf{\underline{Case 2.3: $\lambda=2$.}}
Then (\ref{case222}) becomes
\begin{equation}\label{2222}
\Big(C_1C_3+C_2C_4-\frac{C_1C_2}{4}\Big)\cos2\theta+\Big(C_1C_4-C_2C_3+\frac{C_1^2-C_2^2}{8}\Big)\sin2\theta=0,
\end{equation}
then
\begin{eqnarray}\label{Cstru}
\left\{
  \begin{array}{ll}
    C_1C_3+C_2C_4-\frac{C_1C_2}{4}=0,  \\
   C_1C_4-C_2C_3+\frac{C_1^2-C_2^2}{8}=0.
  \end{array}
\right.
\end{eqnarray}
In this case, we have by (\ref{15ano01}) and (\ref{16b}) that

\begin{eqnarray*}
\left\{
  \begin{array}{ll}
   A=C_3\cos3\theta+C_4\sin3\theta+\frac{1}{8}(C_2\cos\theta-C_1\sin\theta\big),  \\
   L=C_3\sin3\theta-C_4\cos3\theta+\frac{3}{8}(C_1\cos\theta+C_2\sin\theta\big),
  \end{array}
\right.
\end{eqnarray*}
then
\begin{eqnarray}\label{case2221}
\left\{
  \begin{array}{ll}
   v_1=\big(C_3+\frac{C_2}{8}\big)\cos^2\theta+\big(\frac{3C_2}{8}-C_3\big)\sin^2\theta+\big(C_4+\frac{C_1}{8}\big)\sin2\theta,\\
   v_2=\big(C_4-\frac{3C_1}{8}\big)\cos^2\theta-\big(C_4+\frac{C_1}{8}\big)\sin^2\theta-\big(C_3+\frac{C_2}{8}\big)\sin2\theta,
  \end{array}
\right.
\end{eqnarray}
and
\begin{equation*}
  u=Cr^2\left(
       \begin{array}{c}
         v_1 \\
         v_2\\
       \end{array}
     \right)
 =C\left(
       \begin{array}{c}
         \big(C_3+\frac{C_2}{8}\big)x^2+\big(\frac{3C_2}{8}-C_3\big)y^2+\big(2C_4+\frac{C_1}{4}\big)xy\\
         \big(C_4-\frac{3C_1}{8}\big)x^2-\big(C_4+\frac{C_1}{8}\big)y^2-\big(2C_3+\frac{C_2}{4}\big)xy\\
       \end{array}
     \right),
\end{equation*}
with $C_1, C_2, C_3, C_4$ satisfying equation (\ref{Cstru}).
If we replace $\frac{CC_1}{8}$ by $C_1$, $\frac{CC_2}{8}$ by $C_2$, $CC_3$ by $C_3$, $CC_4$ by $C_4$, then
\begin{equation}\label{case2221}
  u=\left(
       \begin{array}{c}
         (C_2+C_3)x^2+(3C_2-C_3)y^2+2(C_1+C_4)xy\\
         (C_4-3C_1)x^2-(C_1+C_4)y^2-2(C_2+C_3)xy\\
       \end{array}
     \right),
\end{equation}
with $C_1, C_2, C_3, C_4$ satisfying
\begin{eqnarray}\label{Cstru-1}
\left\{
  \begin{array}{ll}
    C_1C_3+C_2C_4-2C_1C_2=0,  \\
   C_1C_4-C_2C_3+C_1^2-C_2^2=0.
  \end{array}
\right.
\end{eqnarray}

\textbf{\underline{Case 2.4: $\lambda\neq 0, 1, 2$.}} By (\ref{case222}) we have
\begin{eqnarray*}
\left\{
  \begin{array}{ll}
    C_1C_3+C_2C_4=0, \\
   C_1C_4-C_2C_3=0,\\
   C_1C_2=0, \\
   C_2^2-C_1^2=0,
  \end{array}
\right.
\end{eqnarray*}
namely
\begin{equation}\label{c2.3sc2}
  C_1=C_2=0.
\end{equation}
By (\ref{15ano01}), (\ref{16b}) and (\ref{c2.3sc2}), we have
\begin{eqnarray*}
\left\{
  \begin{array}{ll}
    A=C_3\cos\big((\lambda+1)\theta\big)+C_4\sin\big((\lambda+1)\theta\big), \\
    L=C_3\sin\big((\lambda+1)\theta\big)-C_4\cos\big((\lambda+1)\theta\big),
  \end{array}
\right.
\end{eqnarray*}
then
\begin{eqnarray*}
\left\{
  \begin{array}{ll}
    v_1=C_3\cos(\lambda\theta)+C_4\sin(\lambda\theta),  \\
    v_2=C_4\cos(\lambda\theta)-C_3\sin(\lambda\theta),
  \end{array}
\right.
\end{eqnarray*}
and
\begin{equation}\label{case2222}
  u=Cr^\lambda\left(
       \begin{array}{c}
         C_3\cos(\lambda\theta)+C_4\sin(\lambda\theta) \\
         C_4\cos(\lambda\theta)-C_3\sin(\lambda\theta)\\
       \end{array}
     \right)
 =r^\lambda\left(
       \begin{array}{c}
         C_1'\cos(\lambda\theta)+C_2'\sin(\lambda\theta) \\
         C_2'\cos(\lambda\theta)-C_1'\sin(\lambda\theta) \\
       \end{array}
     \right),
\end{equation}
where $\lambda\neq 0, 1, 2$.\\

Moreover, if $\Omega=\mathbb{R}^2$, it's necessary that $v_i(0)=v_i(2\pi), i=1, 2,$ namely
\begin{eqnarray*}
\left\{
  \begin{array}{ll}
    C_1'[1-\cos(2\lambda\pi)]-C_2'\sin(2\lambda\pi)=0,  \\
    C_2'[1-\cos(2\lambda\pi)]+C_1'\sin(2\lambda\pi)=0.
  \end{array}
\right.
\end{eqnarray*}
Notice that the above equations are linear equations with respect to $1-\cos(2\lambda\pi)$ and $\sin(2\lambda\pi)$, and the determinant
\begin{equation*}\label{dercoef}
\left|\begin{array}{cc}
    C_1' &    -C_2'   \\
    C_2' &    C_1'   \\
\end{array}\right|
=C_1'^2+C_2'^2\neq0,
\end{equation*}
otherwise $u=0$, which is included in the former case (\ref{case221}).
By Cramer's Rule,
$$1-\cos(2\lambda\pi)=\sin(2\lambda\pi)=0,$$
thus $\lambda\in\mathbb{N}$. Recall that $\lambda\neq 0, 1, 2$, then $\lambda\geq3$ and $\lambda\in\mathbb{N}$.

{\bf \underline{Step 3: The pressure expressions.}}
Finally, gathering (\ref{case1}), (\ref{case2}), (\ref{case221}), (\ref{case2221}) and (\ref{case2222}), we obtained that the solution $u$ has five types of forms as in Theorem \ref{mainthm1}.
%
%
%
%
Next we substitute solutions of these types into equations (\ref{SNS}) respectively. All these solutions satisfy
$$\mbox{div}~u=0,$$
then it's left to find the suitable pressure.

In type (i), it is easy to derive the solution
\begin{equation*}
  u=\left(
       \begin{array}{c}
         C_1 \\
         C_2\\
       \end{array}
     \right),~~~
     p=C_3.
\end{equation*}

In type (ii), direct computations gives that
\begin{equation*}
  u=\left(
           \begin{array}{c}
           C_1x+C_2y \\
           C_3x-C_1y \\
           \end{array}
           \right),~~~
 p=-\frac{1}{2}(C_1^2+C_2C_3) (x^2+y^2)+C_4
\end{equation*}
is the solution.

In type (iii), by direct computations, we have that
\begin{equation*}
  \begin{split}
    &\partial_1p-8C_2+2(C_3^2+C_4^2-2C_1C_4+2C_2C_3+C_2^2-3C_1^2)x^3\\
    &~~~~+2(C_3^2+C_4^2+2C_1C_4-2C_2C_3+C_1^2-3C_2^2)xy^2\\
    &~~~~+8(C_1C_3+C_2C_4-2C_1C_2)x^2y=0,\\
   &\partial_2p+8C_1+2(C_3^2+C_4^2+2C_1C_4-2C_2C_3+C_1^2-3C_2^2)y^3\\
    &~~~~+2(C_3^2+C_4^2-2C_1C_4+2C_2C_3+C_2^2-3C_1^2)x^2y\\
    &~~~~+8(C_1C_3+C_2C_4-2C_1C_2)xy^2=0.\\
   \end{split}
\end{equation*}
We apply (\ref{Cstru-1}) to the above equations, then
\begin{equation*}
  \begin{split}
    &\partial_1p-8C_2+2(C_3^2+C_4^2-C_1^2-C_2^2)(x^3+xy^2)=0,\\
   &\partial_2p+8C_1+2(C_3^2+C_4^2-C_1^2-C_2^2)(y^3+x^2y)=0.
   \end{split}
\end{equation*}
Therefore
$$p=\frac{1}{2}(C_1^2+C_2^2-C_3^2-C_4^2)(x^2+y^2)^2+8C_2x-8C_1y+C_5,$$
with $C_1, C_2, C_3, C_4$ satisfying equation (\ref{Cstru-1});

In type (iv), note that $C_1=C_2=0$ in (\ref{c2.3sc2}), which implies $H=0$ and $w\equiv0$ due to (\ref{11}) and (\ref{HAL}). Using
\beno
(u_2,-u_1)^Tw=u\cdot\nabla u-\nabla (\frac{|u|^2}{2}),
\eeno
the pressure can be expressed by $-\frac12 |u|^2+C$.
Then
$$p=-\frac{1}{2}(C_1^2+C_2^2)r^{2\lambda}+C_3.$$

In type (v), $\Omega\neq\mathbb{R}^2$ and
\beno
u=(C_1r+C_2r\ln r)\left(
                              \begin{array}{c}
                                -\sin\theta \\
                                \cos\theta \\
                              \end{array}
                            \right).
\eeno
Direct computations give that
 \beno
 w=-(2C_2\ln r+C_2+2C_1),
 \eeno
and
\beno
u\cdot\nabla u=(C_1+C_2\ln r)\partial_\theta u.
\eeno
Then
\beno
\nabla p&=&\triangle u-u\cdot\nabla u=\nabla^\top w-u\cdot\nabla u\\
&=&\frac{2C_2}{r}\left(
                              \begin{array}{c}
                                -\sin\theta \\
                                \cos\theta \\
                              \end{array}
                            \right)
                            +(C_1+C_2\ln r)^2r\left(
                              \begin{array}{c}
                               \cos\theta  \\
                                 \sin\theta\\
                              \end{array}
                            \right),
\eeno
which implies
\beno
p=\frac12 r^2\Big[C_2^2\ln^2r+(2C_1C_2-C_2^2)\ln r+C_1^2-C_1C_2+\frac{1}{2}C_2^2\Big]+2C_2\theta+C_3
\eeno
by integration by parts.

Thus the proof is complete.
\end{proof}

\section{Proof of Theorem \ref{mainthm3}}

In this section, we prove Theorem \ref{mainthm3}. Our strategy is to reduce the problem here to the one stated in Theorem \ref{mainthm1}.

\begin{proof}[Proof of Theorem \ref{mainthm3}]
In this part, $u$ has the form
$$u=\left(
                              \begin{array}{c}
                                v_1(\theta)\varphi_1(r) \\
                                v_2(\theta)\varphi_2(r) \\
                              \end{array}
                            \right).
$$
In the following, we will only consider $v_1\not \equiv0$, $v_2\not \equiv0$, $\varphi_1\not \equiv0$ and $\varphi_2\not \equiv0$, for the cases $v_1\equiv0$, $v_2\equiv0$,
$\varphi_1\equiv0$ or $\varphi_2\equiv0$ can be reduced to the cases of Theorem \ref{mainthm1}.

Direct computations show that
\begin{equation*}
  \begin{split}
 \mbox{div}~u=&\Big(\cos\theta\partial_r-\frac{\sin\theta}{r} \partial_\theta\Big)(v_1\varphi_1)
  + \Big(\sin\theta\partial_r+\frac{\cos\theta}{r} \partial_\theta\Big)(v_2\varphi_2) \\
=&\cos\theta v_1\varphi_1'+\sin\theta v_2\varphi_2'+\cos\theta v_2'\frac{\varphi_2}{r}-\sin\theta v_1'\frac{\varphi_1}{r},
  \end{split}
\end{equation*}
and
\begin{align}
w&=\partial_2u_1-\partial_1u_2   \nonumber\\
&=\Big(\sin\theta\partial_r+\frac{\cos\theta}{r} \partial_\theta\Big)(v_1\varphi_1)- \Big(\cos\theta\partial_r-\frac{\sin\theta}{r} \partial_\theta\Big)(v_2\varphi_2) \nonumber\\
&=\sin\theta v_1\varphi_1'-\cos\theta v_2\varphi_2'+\cos\theta v_1'\frac{\varphi_1}{r}+\sin\theta v_2'\frac{\varphi_2}{r}. \label{eq:div g-1}
\end{align}
Moreover,
$\mbox{div}~u=0$ implies that
\begin{equation}\label{eq:div}
  \cos\theta v_1\varphi_1'+\sin\theta v_2\varphi_2'+\cos\theta v_2'\frac{\varphi_2}{r}-\sin\theta v_1'\frac{\varphi_1}{r}=0.
\end{equation}
Next we discuss the problem according to whether $\cos\theta v_1$ and $\sin\theta v_2$ are linearly dependent.

\textbf{\underline{Step 1: $\cos\theta v_1$ and $\sin\theta v_2$ are linearly dependent.}}
there exists $\lambda\neq0$, such that
$$\cos\theta v_1=\lambda\sin\theta v_2,$$
since $v_1\not\equiv0$ and $v_2\not\equiv0$. Then
\begin{equation}\label{gsvc1-1}
v_1=\lambda\tan\theta v_2.
\end{equation}
Substituting (\ref{gsvc1-1}) into (\ref{eq:div}), we obtain that
%
\begin{equation}\label{eq:div1}
 \sin\theta v_2(\lambda\varphi_1'+\varphi_2')+\cos\theta v_2'\frac{\varphi_2}{r}
-\lambda\sin\theta(\sec^2\theta v_2+\tan\theta v_2')\frac{\varphi_1}{r}=0.
\end{equation}
Since $v_2\not\equiv0$, there exist an interval $K$, such that $v_2, \sin\theta, \cos\theta\neq0$ when $\theta\in K.$
Then we deduce from (\ref{eq:div1}) that
$$
\lambda\varphi_1'+\varphi_2'+\frac{v_2'}{\tan\theta v_2}\frac{\varphi_2}{r}
-\lambda \Big(\sec^2\theta+\frac{\tan\theta v_2'}{v_2}\Big)\frac{\varphi_1}{r}=0,~~~~~\theta\in K.
$$
Let $M(\theta)=\frac{v_2'}{\tan\theta v_2}$, $N(\theta)=\sec^2\theta+M(\theta)\tan^2\theta$,
then we can rewrite the above formula as
\begin{equation}\label{gsvc1-1a}
\lambda\varphi_1'+\varphi_2'+M(\theta)\frac{\varphi_2}{r}
-\lambda N(\theta)\frac{\varphi_1}{r}=0,~~~~~\theta\in K.
\end{equation}

\textbf{Case 1.1}: If $M(\theta)$ is not a constant in $K$, namely there exists $\theta_1, \theta_2\in K$, such that
$$M(\theta_1)\neq M(\theta_2),$$
then
$$\lambda\varphi_1'+\varphi_2'+M(\theta_1)\frac{\varphi_2}{r}-\lambda N(\theta_1)\frac{\varphi_1}{r}=0,$$
$$\lambda\varphi_1'+\varphi_2'+M(\theta_2)\frac{\varphi_2}{r}-\lambda N(\theta_2)\frac{\varphi_2}{r}=0.$$
In the above two equations, the first one minus the second gives that
$$\varphi_2=\lambda\frac{N(\theta_1)-N(\theta_2)}{M(\theta_1)-M(\theta_2)}\varphi_1=:C\varphi_1,$$
then $u=\left(
            \begin{array}{c}
           v_1\varphi_1 \\
           Cv_2 \varphi_1\\
            \end{array}
            \right)
=\left(
            \begin{array}{c}
            v_1\\\
           Cv_2\\
            \end{array}
            \right)
\varphi_1,$
reducing to the cases of Theorem \ref{mainthm1}.

\textbf{Case 1.2}: If $N(\theta)$ is not a constant in $K$, this case is similar with the above case of \textbf{Case 1.1}.

\textbf{Case 1.3}: If both $M(\theta)$ and $N(\theta)$ are constants in $K$, it is easy to see that
$M(\theta)=-1, N(\theta)=1,$
and (\ref{gsvc1-1a}) becomes
$$(\lambda\varphi_1+\varphi_2)'=\frac{\lambda\varphi_1+\varphi_2}{r},$$
then
$\lambda\varphi_1+\varphi_2=C_1r,$
and
\begin{equation}\label{phy1phy2.rela}
\varphi_2=C_1r-\lambda\varphi_1.
\end{equation}
Now we substitute (\ref{phy1phy2.rela}) to (\ref{eq:div1}) and obtain that
$$(\sin\theta v_2+\cos\theta v_2')\Big(C_1-\lambda\sec^2\theta\frac{\varphi_1}{r}\Big)=0.$$
Since $\lambda\neq0$ and $\varphi_1\not\equiv0$, then we  have
$\sin\theta v_2+\cos\theta v_2'=0.$
Applying Lemma \ref{lemma2new}, there holds
$v_2=C_2\cos\theta$
with $C_2\neq0$. By (\ref{gsvc1-1}), we get
$v_1=\lambda C_2\sin\theta.$
We substitute
$$\left\{
  \begin{array}{ll}
    v_1=\lambda C_2\sin\theta \\
    v_2=C_2\cos\theta \\
\varphi_2=C_1r-\lambda\varphi_1
  \end{array}
\right.$$
into (\ref{eq:div g-1}) and obtain that
$$w=\lambda C_2\big(\varphi_1'+\frac{\varphi_1}{r}\big)-C_1C_2,$$
then
\begin{equation}\label{gsvc1-w1}
\begin{split}
 \Delta w=&\big(\partial_r^2+\frac{1}{r}\partial_r+\frac{1}{r^2}\partial_\theta^2\big)w=\lambda C_2\big(\partial_r^2+\frac{1}{r}\partial_r\big)\big(\varphi_1'+\frac{\varphi_1}{r}\big)\\
  =&\lambda C_2\big(\varphi_1'''+\frac{2\varphi_1''}{r}-\frac{\varphi_1'}{r^2}+\frac{\varphi_1}{r^3}\big);\\
  \partial_1w=&\lambda C_2\big(\varphi_1'+\frac{\varphi_1}{r}\big)'\cos \theta=\lambda C_2\cos\theta\big(\varphi_1''+\frac{\varphi_1'}{r}-\frac{\varphi_1}{r^2}\big);\\
  \partial_2w=&\lambda C_2\sin\theta\big(\varphi_1''+\frac{\varphi_1'}{r}-\frac{\varphi_1}{r^2}\big);\\
  u\cdot \nabla w=&v_1\varphi_1\partial_1w+v_2\varphi_2\partial_2w=\lambda C_1C_2^2\sin\theta\cos\theta r\big(\varphi_1''+\frac{\varphi_1'}{r}-\frac{\varphi_1}{r^2}\big).
\end{split}
\end{equation}
Combine (\ref{weqn}) and (\ref{gsvc1-w1}), and notice that $C_2\neq0$, $\lambda\neq0$, then we obtain
$$\varphi_1'''+\frac{2\varphi_1''}{r}-\frac{\varphi_1'}{r^2}+\frac{\varphi_1}{r^3}=C_1C_2\sin\theta\cos\theta r\big(\varphi_1''+\frac{\varphi_1'}{r}-\frac{\varphi_1}{r^2}\big),$$
which implies that
\begin{equation}\label{phy1.eqn}
  \left\{
    \begin{array}{ll}
      C_1\big(\varphi_1''+\frac{\varphi_1'}{r}-\frac{\varphi_1}{r^2}\big)=0,\\
      \varphi_1'''+\frac{2\varphi_1''}{r}-\frac{\varphi_1'}{r^2}+\frac{\varphi_1}{r^3}=0.
    \end{array}
  \right.
\end{equation}
Notice that the second equation of (\ref{phy1.eqn}) is namely (\ref{phieqn}), whose solutions are
$\varphi_1=C_3r+ C_4r\ln r.$
These solutions verify the first equation of (\ref{phy1.eqn}) if and only if $C_4=0$, then $\varphi_1=C_3r$ are the solutions of (\ref{phy1.eqn}).
Therefore,
\begin{equation*}
  \left\{
    \begin{array}{ll}
    \varphi_1=C_3r, \\
    \varphi_2=(C_1-\lambda C_3)r.
  \end{array}
  \right.
\end{equation*}
Finally we have
$$u=C_2\left(
            \begin{array}{c}
           \lambda C_3 r\sin\theta \\
           (C_1-\lambda C_3)r\cos\theta\\
            \end{array}
            \right)
=\left(
            \begin{array}{c}
            C_4y\\
           C_5x\\
            \end{array}
            \right),
~~~p=-\frac{1}{2}C_4C_5(x^2+y^2)+C_6,$$
included in the type (ii) as shown in Theorem \ref{mainthm1}.

\textbf{\underline{Step 2:  $\cos\theta v_1$ and $\sin\theta v_2$ are linearly independent.}} Then there exist $\theta_1\neq\theta_2$
such that the determinant
\begin{equation*}\label{dercoef}
\mathbf{D}_1:= \left|\begin{array}{cc}
    \cos\theta_1 v_1(\theta_1) &    \sin\theta_1 v_2(\theta_1)   \\
    \cos\theta_2 v_1(\theta_2) &    \sin\theta_2 v_2(\theta_2)   \\
\end{array}\right|\\
\neq0.
\end{equation*}
We take $\theta=\theta_1$ and $\theta=\theta_2$ respectively in (\ref{eq:div}), then we obtain equations
\begin{equation}\label{keyidea}
\left\{
    \begin{array}{ll}
      \cos\theta_1 v_1(\theta_1)\varphi_1'+\sin\theta_1 v_2(\theta_1)\varphi_2'=\sin\theta_1 v_1'(\theta_1)\frac{\varphi_1}{r}-\cos\theta_1 v_2'(\theta_1)\frac{\varphi_2}{r}, \\
      \cos\theta_2 v_1(\theta_2)\varphi_1'+\sin\theta_2 v_2(\theta_2)\varphi_2'=\sin\theta_2 v_1'(\theta_2)\frac{\varphi_1}{r}-\cos\theta_2 v_2'(\theta_2)\frac{\varphi_2}{r}.
    \end{array}
  \right.
\end{equation}
Notice that the determinant of the coefficients of (\ref{keyidea}) is exactly $\mathbf{D}_1$. Since $\mathbf{D}_1\neq0$, by Cramer's Rule we must have
\begin{equation}\label{phy1phy2}
\left\{
    \begin{array}{ll}
      \varphi_1'=a\frac{\varphi_1}{r}+b\frac{\varphi_2}{r}, \\
      \varphi_2'=c\frac{\varphi_1}{r}+d\frac{\varphi_2}{r}.
    \end{array}
  \right.
\end{equation}
We substitute (\ref{phy1phy2}) into (\ref{eq:div}) and obtain that
$$(a\cos\theta v_1+c\sin\theta v_2-\sin\theta v_1')\varphi_1+(b\cos\theta v_1+d\sin\theta v_2+\cos\theta v_2')\varphi_2=0.$$
Since $\varphi_2\not\equiv0$, by Lemma \ref{lemma1} we get
either
\begin{equation}\label{case2.1a}
\left\{
  \begin{array}{ll}
   \varphi_1=\lambda\varphi_2 \\
    b\cos\theta v_1+d\sin\theta v_2+\cos\theta v_2'=-\lambda(a\cos\theta v_1+c\sin\theta v_2-\sin\theta v_1')
  \end{array}
\right.
\end{equation}
or
\begin{equation}\label{v1v2cond}
\left\{
  \begin{array}{ll}
    a\cos\theta v_1+c\sin\theta v_2-\sin\theta v_1'=0, \\
    b\cos\theta v_1+d\sin\theta v_2+\cos\theta v_2'=0.
  \end{array}
\right.
\end{equation}
In the first case (\ref{case2.1a}), since $\varphi_1=\lambda\varphi_2$,
$u=\left(
            \begin{array}{c}
           \lambda v_1\varphi_2 \\
           v_2\varphi_2 \\
            \end{array}
            \right)
=\left(
            \begin{array}{c}
          \lambda v_1 \\
            v_2\\
            \end{array}
            \right)
\varphi_2$
can be reduced to the cases of Theorem \ref{mainthm1}.

Next we focus on the second case of (\ref{v1v2cond}), which implies that
\begin{equation}\label{v1v2deri}
\left\{
  \begin{array}{ll}
  v_1'=a\cot\theta v_1+c v_2, \\
  v_2'=-(bv_1+d\tan\theta v_2).
  \end{array}
\right.
\end{equation}
Substituting (\ref{phy1phy2}) and (\ref{v1v2deri}) into (\ref{eq:div g-1}), we obtain that
\begin{equation}\label{simplyw}
  w=\frac{av_1}{\sin \theta}\frac{\varphi_1}{r}-\frac{dv_2}{\cos \theta}\frac{\varphi_2}{r}.
\end{equation}
$w$ satisfy the equation $\triangle w=u\cdot \nabla w$. First, we compute $u\cdot \nabla w$.
\begin{align}
 u\cdot \nabla w&=v_1\varphi_1\partial_1w+v_2\varphi_2\partial_2w \nonumber\\
 &=v_1\varphi_1\big(\cos\theta\partial_r-\frac{\sin\theta}{r} \partial_\theta\big)w+v_2\varphi_2\big(\sin\theta\partial_r+\frac{\cos\theta}{r} \partial_\theta\big)w   \nonumber\\
&=(\cos\theta v_1\varphi_1+\sin\theta v_2\varphi_2)\partial_r w+\frac{1}{r}(-\sin\theta v_1\varphi_1+\cos\theta v_2\varphi_2)\partial_\theta w.    \label{udotw}
\end{align}
According to (\ref{simplyw}) and (\ref{phy1phy2}),
\begin{align}
\partial_r w=&\frac{av_1}{\sin \theta}\Big(\frac{\varphi_1}{r}\Big)'
   -\frac{dv_2}{\cos \theta}\Big(\frac{\varphi_2}{r}\Big)'          \nonumber\\
=&\frac{av_1}{\sin \theta}r^{-2}\big[(a-1)\varphi_1+b\varphi_2\big]
-\frac{dv_2}{\cos \theta}r^{-2}\big[c\varphi_1+(d-1)\varphi_2\big]  \nonumber\\
=&\Big[a(a-1)\frac{v_1}{\sin\theta}-cd\frac{v_2}{\cos \theta}\Big]r^{-2}\varphi_1
+\Big[ab\frac{v_1}{\sin\theta}-d(d-1)\frac{v_2}{\cos \theta}\Big]r^{-2}\varphi_2. \label{partialr-w}
\end{align}
According to (\ref{simplyw}) and (\ref{v1v2deri}),
\begin{align}
\partial_\theta w
=&\frac{a\varphi_1}{r}\Big(\frac{v_1}{\sin\theta}\Big)'
     -\frac{d\varphi_2}{r}\Big(\frac{v_2}{\cos\theta}\Big)'              \nonumber \\
=&\frac{a}{\sin\theta}\big[(a-1)\cot\theta v_1+cv_2\big]r^{-1}\varphi_1
+\frac{d}{\cos\theta}\big[bv_1+(d-1)\tan\theta v_2\big]r^{-1}\varphi_2.  \label{partialth-w}
\end{align}
Substituting (\ref{partialr-w}) and (\ref{partialth-w}) into (\ref{udotw}), we obtain that
\begin{equation}\label{udotw-1}
u\cdot \nabla w=r^{-2}\big[-c(a+d)v_1v_2\varphi_1^2+b(a+d)v_1v_2\varphi_2^2
+F_3\varphi_1\varphi_2\big],
\end{equation}
where
\begin{equation}\label{F3}
F_3=(a\cot\theta-d\tan\theta)(bv_1^2+cv_2^2)+\Big(\frac{a^2-a}{\sin^2\theta}
-\frac{d^2-d}{\cos^2\theta}\Big)v_1v_2.
\end{equation}
Next we compute $\Delta w$.
\begin{equation}\label{Deltaw-1}
  \begin{split}
 \Delta w=&\big(\partial_r^2+r^{-1}\partial_r+r^{-2}\partial_\theta^2\big)w.\\
 \end{split}
\end{equation}
By (\ref{partialr-w}) and (\ref{phy1phy2}),
\begin{align}
\partial_r^2 w
=&\Big\{\big[a(a-1)(a-2)+abc\big]\frac{v_1}{\sin\theta}-cd(a+d-3)\frac{v_2}{\cos\theta} \Big\}r^{-3}\varphi_1\nonumber\\
&+\Big\{ab(a+d-3)\frac{v_1}{\sin\theta}-\big[d(d-1)(d-2)
+bcd\big]\frac{v_2}{\cos\theta} \Big\}r^{-3}\varphi_2\label{partial2r-w}
\end{align}
By (\ref{partialth-w}) and (\ref{v1v2deri}),
\begin{equation}\label{partial2th-w}
\begin{split}
\partial_\theta^2w
=\Big\{\big[a(a-1)(a-2)\cot^2\theta-a(bc+a-1)\big]\frac{v_1}{\sin\theta}
+ac\big[(a-2)\cot^2\theta-d\big]\frac{v_2}{\cos\theta}\Big\}r^{-1}\varphi_1\\
-\Big\{\big[d(d-1)(d-2)\tan^2\theta-d(bc+d-1)\big]\frac{v_2}{\cos\theta}
+bd\big[(d-2)\tan^2\theta-a\big]\frac{v_1}{\sin\theta}\Big\}r^{-1}\varphi_2\\
\end{split}
\end{equation}
Substituting (\ref{partialr-w}), (\ref{partial2r-w}) and (\ref{partial2th-w}) into (\ref{Deltaw-1}), we obtain that
\begin{equation}\label{Deltaw-2}
\Delta w=r^{-3}(F_1\varphi_1-F_2\varphi_2),
\end{equation}
where
\begin{equation}\label{F1F2}
  \begin{split}
F_1=&a(a-1)(a-2)\frac{v_1}{\sin^3\theta}+c\big[a(a-2)\cot^2\theta-d(d+2a-2)\big]\frac{v_2}{\cos\theta},\\
F_2=&d(d-1)(d-2)\frac{v_2}{\cos^3\theta}+b\big[d(d-2)\tan^2\theta-a(a+2d-2)\big]\frac{v_1}{\sin\theta}.
  \end{split}
\end{equation}
$\triangle w=u\cdot \nabla w$, (\ref{udotw-1}) and (\ref{Deltaw-2}) yield that
\begin{equation}\label{criti.eqn}
F_1\varphi_1-F_2\varphi_2=(a+d)v_1v_2r\big(b\varphi_2^2-c\varphi_1^2\big)+F_3r\varphi_1\varphi_2,
\end{equation}
where $F_1, F_2, F_3$ are given by (\ref{F3}) and (\ref{F1F2}), and $v_1, v_2$ satisfy (\ref{v1v2cond}).

Since $\varphi_1, \varphi_2$ satisfy equations (\ref{phy1phy2}),  applying Lemma \ref{lemma3} for them,
 we obtain their expressions as shown in Lemma \ref{lemma3}.
If $\varphi_1$ and $\varphi_2$ are linearly dependent, one can deduce this problem to the previous one,
hence it suffices to consider that they are linearly independent.

\textbf{\underline{Case (1): $b=0$, $d=a$.}} At this time,
\begin{equation*}
\left\{
    \begin{array}{ll}
\varphi_1=C_1r^a, \\
\varphi_2=(cC_1\ln r+C_2)r^a.
    \end{array}
  \right.
\end{equation*}
Since $\varphi_1$ and $\varphi_2$ are linearly independent, then $cC_1\neq0$. $\varphi_1, \varphi_2\in C^0([0,\infty))$, then $a>0$.

Substitute $b=0$, $d=a$ and the expressions of $\varphi_1$ and $\varphi_2$ into (\ref{criti.eqn}),
\begin{equation*}
  (C_1F_1-C_2F_2)-cC_1F_2\ln r=r^{a+1}\big[cC_1^2F_3\ln r -2acC_1^2 v_1v_2+C_1C_2F_3\big],
\end{equation*}
which implies
\begin{equation*}\label{case1crieqn1}
\left\{
  \begin{array}{ll}
C_1F_1-C_2F_2=0;\\
cC_1F_2=0;\\
cC_1^2F_3=0;\\
-2acC_1^2v_1v_2+C_1C_2F_3=0.
  \end{array}
\right.
\end{equation*}
Since $cC_1\neq0$ and $a>0$, then
$F_1=F_2=F_3=v_1v_2\equiv0.$
Applying Lemma \ref{lemma4}, we have
$v_2\equiv0,$
which contradicts our assumption that $v_2\not\equiv0$.
Therefore this case doesn't exist.

\textbf{\underline{Case (2): $b=0$, $d\neq a$.}} We have
\begin{equation*}
\left\{
    \begin{array}{ll}
      \varphi_1=C_1r^a, \\
      \varphi_2=\frac{c}{a-d}C_1r^a+C_2r^d.
    \end{array}
  \right.
\end{equation*}
Since $\varphi_1$ and $\varphi_2$ are linearly independent, then $C_1C_2\neq0$. $\varphi_1, \varphi_2\in C^0$, then $a\geq 0$ and  $d\geq0$.
$d\neq a$ implies that $a+d>0$.

In this case, equation (\ref{criti.eqn}) becomes
\begin{equation}\label{caseIcrieqn}
C_1\Big(F_1-\frac{c F_2}{a-d}\Big)r^a-C_2F_2r^d=cC_1^2\Big[\frac{F_3}{a-d}-(a+d)v_1v_2\Big]r^{2a+1}
+C_1C_2F_3r^{a+d+1}.
\end{equation}
$r^a, r^d, r^{2a+1}, r^{a+d+1}$ appear in the above formula. Obviously, $r^a$ is different from the other three,
so is $r^{a+d+1}$. Then the coefficients of $r^a$ and $r^{a+d+1}$ must be 0, which give that
\begin{numcases}{}
F_1=\frac{c}{a-d}F_2;    \label{case2.eqn1} \\
F_3=0;                   \label{case2.eqn2} \\
C_2F_2r^d=c(a+d)C_1^2v_1v_2r^{2a+1}.    \label{case2.eqn3}
\end{numcases}
This situation will be divided into two subcases for further discussion.

\textbf{\underline{Case (2.1):  $c=0$.}} The above equations are reduced to
$$F_1=F_2=F_3=0,$$
namely
\begin{equation}\label{Case2.2.2.I.1}
\left\{
  \begin{array}{ll}
a(a-1)(a-2)\frac{v_1}{\sin^3\theta}=0;\\
d(d-1)(d-2)\frac{v_2}{\cos^3\theta}=0;\\
\Big(\frac{a^2-a}{\sin^2\theta}-\frac{d^2-d}{\cos^2\theta}\Big)v_1v_2 =0.
  \end{array}
\right.
\end{equation}
Since $v_1\not\equiv0$ and $v_2\not\equiv0$, then the first and second equation of (\ref{Case2.2.2.I.1}) give that
$$a=0,1,2;~~~~~d=0,1,2.$$
Notice that $c=0$, then the first equation of (\ref{v1v2cond}) is reduced to
\begin{equation*}\label{v1v2cond1}
    a\cos\theta v_1-\sin\theta v_1'=0.
\end{equation*}
Since $a=0,1,2$ and $v_1\in C^2$, we use a argument similar to the one in Lemma \ref{lemma2new}, then we obtain that
\begin{equation}\label{v1v2cond2}
    v_1=C_3\sin^a\theta,~~~~~C_3\neq0,
\end{equation}
Similarly, since $b=0$, the second equation of (\ref{v1v2cond}) is reduced to
\begin{equation*}
    d\sin\theta v_2+\cos\theta v_2'=0,
\end{equation*}
Since $d=0,1,2$ and $v_2\in C^2$, we apply Lemma \ref{lemma2new} and obtain that
\begin{equation}\label{v1v2cond2a'}
    v_2=C_4\cos^d\theta,~~~~~C_4\neq0.
\end{equation}
Having (\ref{v1v2cond2}) and (\ref{v1v2cond2a'}), one can reduce the third equation of (\ref{Case2.2.2.I.1}) to
\begin{equation*}
\left\{
  \begin{array}{ll}
a^2-a=0;\\
d^2-d=0.\\
  \end{array}
\right.
\end{equation*}
Since $d\neq a$, the above equations have solutions
\begin{equation*}
  \left\{
     \begin{array}{ll}
a=1  \\
d=0~;
     \end{array}
   \right.~~~~
  \left\{
     \begin{array}{ll}
a=0 \\
d=1~.
     \end{array}
   \right.
\end{equation*}

If $a=1, d=0$, then
$$\varphi_1=C_1r,~~~\varphi_2=C_2,~~~v_1=C_3\sin\theta,~~~v_2=C_4,$$
and
\begin{equation}\label{extrasolu1}
u=\left(
       \begin{array}{c}
         C_1C_3r\sin\theta \\
         C_2C_4\\
       \end{array}
     \right)
=\left(
       \begin{array}{c}
         C_5y \\
         C_6\\
       \end{array}
     \right),~~~
     p=-C_5C_6x+C_7.
\end{equation}
If $a=0, d=1$, then
$$\varphi_1=C_1,~~~\varphi_2=C_2r,~~~v_1=C_3,~~~v_2=C_4\cos\theta,$$
and
\begin{equation}\label{extrasolu2}
u=\left(
       \begin{array}{c}
         C_1C_3 \\
         C_2C_4r\cos\theta\\
       \end{array}
     \right)
=\left(
       \begin{array}{c}
         C_5 \\
         C_6x\\
       \end{array}
     \right),~~~
     p=-C_5C_6y+C_7.
\end{equation}

\textbf{\underline{Case (2.2):  $c\neq0$.}} Consider equation (\ref{case2.eqn3}).\\
If $d\neq2a+1$, $$C_2F_2=c(a+d)C_1^2v_1v_2\equiv0,$$
namely
$$F_2\equiv0,~~~~~v_1v_2\equiv0.$$
According to Lemma \ref{lemma4},
$v_2\equiv0,$
which contradicts that $v_2\not\equiv0$. Therefore
$$d=2a+1.$$

With (\ref{F1F2}) and $b=0$, we rewrite (\ref{case2.eqn1}) as
\begin{equation}\label{v1v2rela1a}
  \begin{split}
a(a-1)(a-2)v_1=&\frac{c\sin\theta}{\cos^3\theta}\Big[\frac{d(d-1)(d-2)}{a-d}\sin^2\theta-a(a-2)\cos^4\theta\\
&~~~~~~~~~~~~~~~~~~~+d(d+2a-2)\sin^2\theta\cos^2\theta\Big]v_2.
  \end{split}
\end{equation}

If $a(a-1)(a-2)=0$, then
$$\Big[\frac{d(d-1)(d-2)}{a-d}\sin^2\theta-a(a-2)\cos^4\theta+d(d+2a-2)\sin^2\theta\cos^2\theta\Big]v_2=0.$$
Since $v_2\not\equiv0$, there holds
\begin{equation*}
  \left\{
    \begin{array}{ll}
d(d-1)(d-2)=0;\\
a(a-2)=0;\\
d(d+2a-2)=0.
    \end{array}
  \right.
\end{equation*}
The above equations have solutions
\begin{equation*}
  \left\{
     \begin{array}{ll}
a=0  \\
d=0~;
     \end{array}
   \right.~~~~
  \left\{
     \begin{array}{ll}
a=0  \\
d=2~;
     \end{array}
   \right.~~~~
  \left\{
     \begin{array}{ll}
a=2 \\
d=0~.
     \end{array}
   \right.
\end{equation*}
These all contradict that $d=2a+1$. Therefore
$$a(a-1)(a-2)\neq0.$$
This and (\ref{v1v2rela1a}) give that
\begin{equation}\label{v1v2rela3a}
  \begin{split}
v_1=\frac{c}{a(a-1)(a-2)}\frac{\sin\theta}{\cos^3\theta}&\Big[\frac{d(d-1)(d-2)}{a-d}\sin^2\theta-a(a-2)\cos^4\theta\\
&~~~~~~+d(d+2a-2)\sin^2\theta\cos^2\theta\Big]v_2.
  \end{split}
\end{equation}
Notice that when $b=0$, equation (\ref{case2.eqn2}) is
$$F_3=c(a\cot\theta-d\tan\theta)v_2^2+\Big(\frac{a^2-a}{\sin^2\theta}
-\frac{d^2-d}{\cos^2\theta}\Big)v_1v_2=0.$$
Since $v_2\not\equiv0$, there exists an interval $K_1$, such that
$$v_2\neq0,~~~~~\theta\in K_1.$$
Then
$$
c(a\cot\theta-d\tan\theta)v_2+\Big(\frac{a^2-a}{\sin^2\theta}
-\frac{d^2-d}{\cos^2\theta}\Big)v_1=0,~~~~~\theta\in K_1.
$$
Substituting (\ref{v1v2rela3a}) into the above equation, we obtain that in $K_1$,
\ben\label{eq:a-d}
&&a(a-1)(a-2)(a\cos^2\theta-d\sin^2\theta)\cos^4\theta+\big[(a^2-a)\cos^2\theta-(d^2-d)\sin^2\theta\big]\nonumber\\
&&~~~~~\times\Big[\frac{d(d-1)(d-2)}{a-d}\sin^2\theta-a(a-2)\cos^4\theta+d(d+2a-2)\sin^2\theta\cos^2\theta\Big]=0.\nonumber\\
  \een
Notice that the left hand side of the above equation is a polynomial with respect to $\cos^2\theta$,
and the constant term of this polynomial is $\frac{d^2(d-1)^2(d-2)}{d-a}$, then we have
$$\frac{d^2(d-1)^2(d-2)}{d-a}=0,$$
and
$d=0,~1,~2.$ Recall that $d=2a+1$, $a\geq 0$ and $a(a-1)(a-2)\neq0$, then
$$d=2,~~~~a=\frac12.$$
Substituting this into (\ref{eq:a-d}), we obtain that $-\frac{11}{16}\cos^2\theta=\sin^2\theta$, which is impossible.
Therefore, Case (2.2) does't exist.

\textbf{\underline{Case (3): $b\neq0$, $\delta>0$.}} Now we have
\begin{equation*}\label{phy1phy2caseII}
\left\{
    \begin{array}{ll}
     \varphi_1=C_1r^m+C_2r^n, \\
     \varphi_2=\frac{m-a}{b}C_1r^m+\frac{n-a}{b}C_2r^n,
    \end{array}
  \right.
\end{equation*}
where $m, n$ are two different real roots of equation (\ref{char.eqn}) and $m>n$.

Since $\varphi_1$ and $\varphi_2$ are linearly independent, then $C_1C_2\neq0$. $\varphi_1, \varphi_2\in C^0$, then $m>n\geq 0.$

By Vieta's theorem,
\begin{equation}\label{wietathm}
\begin{split}
a+d=&m+n>0,\\
ad-bc=&mn.
\end{split}
\end{equation}

In this case, equation (\ref{criti.eqn}) becomes
\begin{equation}\label{caseIIcrieqn}
\begin{split}
&C_1\big[bF_1-(m-a)F_2\big]r^m+C_2\big[bF_1-(n-a)F_2\big]r^n\\
=&C_1^2\Big\{(a+d)\big[(m-a)^2-bc\big]v_1v_2+(m-a)F_3\Big\}r^{2m+1}\\
&+C_2^2\Big\{(a+d)\big[(n-a)^2-bc\big]v_1v_2+(n-a)F_3\Big\}r^{2n+1}\\
&+C_1C_2\Big\{2(a+d)\big[(m-a)(n-a)-bc\big]v_1v_2+(m+n-2a)F_3\Big\}r^{m+n+1}.
\end{split}
\end{equation}
Note that $r^m, r^n, r^{2m+1}, r^{2n+1}, r^{m+n+1}$ appear in the above equation. Since $m>n$,
$$2m+1>m+n+1>2n+1, m>n,$$
the coefficients of $r^{2m+1}$, $r^{m+n+1}$ and $r^n$ must be $0$, which implies that
\begin{equation}\label{caseIIcrieqn1}
\left\{
  \begin{array}{ll}
(a+d)\big[(m-a)^2-bc\big]v_1v_2+(m-a)F_3=0;\\
2(a+d)\big[(m-a)(n-a)-bc\big]v_1v_2+(m+n-2a)F_3=0;\\
bF_1=(n-a)F_2;\\
C_1\big[bF_1-(m-a)F_2\big]r^m=C_2^2\Big\{(a+d)\big[(n-a)^2-bc\big]v_1v_2+(n-a)F_3\Big\}r^{2n+1}.
  \end{array}
\right.
\end{equation}
Notice that $m, n$ are roots of equation (\ref{char.eqn}). This and (\ref{wietathm}) give that
\begin{equation}\label{tip1}
\begin{split}
(m-a)^2-bc=&(d-a)(m-a);\\
(n-a)^2-bc=&(d-a)(n-a);\\
(m-a)(n-a)=&-bc;\\
m+n=&a+d.
\end{split}
\end{equation}
Substitute the third equation of (\ref{caseIIcrieqn1}) into the forth one and apply (\ref{tip1}), then one can rewrite (\ref{caseIIcrieqn1}) as
\begin{equation}\label{caseIIcrieqn2}
\left\{
  \begin{array}{ll}
(a+d)(d-a)(m-a)v_1v_2+(m-a)F_3=0;  \\
4bc(a+d)v_1v_2+(a-d)F_3=0;                  \\
F_1=\frac{n-a}{b}F_2;                      \\
\big[(a+d)(d-a)(n-a)v_1v_2+(n-a)F_3\big]r^{2n+1}=(n-m)C_1C_2^{-2}F_2r^m.
  \end{array}
\right.
\end{equation}
Consider the first and second equation of (\ref{caseIIcrieqn2}), where
the determinant of the coefficients
\begin{equation*}
  \begin{split}
  \mathbf{D}_2:=&\left|\begin{array}{cc}
    (a+d)(d-a)(m-a) &    m-a     \\
    4bc(a+d) &    a-d   \\
\end{array}\right|\\
=&(a+d)(a-m)\big[(d-a)^2+4bc\big]=(a+d)(a-m)\delta.
  \end{split}
\end{equation*}
If $\mathbf{D}_2\neq0$, by Cramer's Rule,
$$v_1v_2\equiv0,~~~F_3\equiv0.$$
According to Lemma \ref{lemma4},
$$v_1\equiv0,$$
which contradicts that $v_1\not\equiv0$. Then we must have $\mathbf{D}_2=0$, namely
$$
(a+d)(a-m)\delta=0.
$$
Notice that $a+d>0$ and $\delta>0$, then
$$m=a.$$
Therefore, $n=a+d-m=d$, $bc=ad-mn=0$. Since $b\neq0$, then $c=0$. In summary,
\begin{equation}\label{tip2}
  m=a,~~~ n=d,~~~c=0,
\end{equation}
and $a>d\geq 0.$

With (\ref{tip2}), we can reduce equations (\ref{caseIIcrieqn2}) to
\begin{numcases}{}
F_1=\frac{d-a}{b}F_2;                   \label{caseII.eqn2} \\
F_3=0;                                  \label{caseII.eqn1} \\
(a+d)(d-a)v_1v_2r^{2d+1}=C_1C_2^{-2}F_2r^a.    \label{caseII.eqn3}
\end{numcases}
The following argument for this case is very similar to that for Case (2.2). For completeness, let's briefly describe the proof.

If $a\neq2d+1$, equation (\ref{caseII.eqn3}) is equivalent to
$$(a+d)(d-a)v_1v_2=C_1C_2^{-2}F_2\equiv0,$$
and then $v_1v_2\equiv0$, which contradicts that $v_1\not\equiv0$. Therefore
$$a=2d+1.$$

With (\ref{F1F2}) and $c=0$, we rewrite (\ref{caseII.eqn2}) as
\begin{equation}\label{v1v2rela1}
  \begin{split}
d(d-1)(d-2)v_2=&\frac{b\cos\theta}{\sin^3\theta}\Big[\frac{a(a-1)(a-2)}{d-a}\cos^2\theta-d(d-2)\sin^4\theta\\
&~~~~~~~~~~~~~~~~~~~+a(a+2d-2)\sin^2\theta\cos^2\theta\Big]v_1.
  \end{split}
\end{equation}
If $d(d-1)(d-2)=0$, then
$$\Big[\frac{a(a-1)(a-2)}{d-a}\cos^2\theta-d(d-2)\sin^4\theta+a(a+2d-2)\sin^2\theta\cos^2\theta\Big]v_1=0.$$
Since $v_1\not\equiv0$, we must have
\begin{equation*}
  \left\{
    \begin{array}{ll}
a(a-1)(a-2)=0;\\
d(d-2)=0;\\
a(a+2d-2)=0.
    \end{array}
  \right.
\end{equation*}
The above equations have solutions
\begin{equation*}
  \left\{
     \begin{array}{ll}
d=0  \\
a=0~;
     \end{array}
   \right.~~~~
  \left\{
     \begin{array}{ll}
d=0  \\
a=2~;
     \end{array}
   \right.~~~~
  \left\{
     \begin{array}{ll}
d=2 \\
a=0~.
     \end{array}
   \right.
\end{equation*}
These all contradict that $a=2d+1$. Therefore
$$d(d-1)(d-2)\neq0.$$
This and (\ref{v1v2rela1}) give that
\begin{equation}\label{v1v2rela3}
  \begin{split}
v_2=\frac{b}{d(d-1)(d-2)}\frac{\cos\theta}{\sin^3\theta}&\Big[\frac{a(a-1)(a-2)}{d-a}\cos^2\theta-d(d-2)\sin^4\theta\\
&~~~~~~+a(a+2d-2)\sin^2\theta\cos^2\theta\Big]v_1.
  \end{split}
\end{equation}
Notice that when $c=0$, equation (\ref{caseII.eqn1}) is
$$F_3=b(a\cot\theta-d\tan\theta)v_1^2+\Big(\frac{a^2-a}{\sin^2\theta}
-\frac{d^2-d}{\cos^2\theta}\Big)v_1v_2=0.$$
Since $v_1\not\equiv0$, there exists an interval $K_2$, such that
$$v_1\neq0,~~~~~\theta\in K_2.$$
Then
$$
b(a\cot\theta-d\tan\theta)v_1+\Big(\frac{a^2-a}{\sin^2\theta}
-\frac{d^2-d}{\cos^2\theta}\Big)v_2=0,~~~~~\theta\in K_2.
$$
Substituting (\ref{v1v2rela3}) into the above equation, we obtain that in $K_2$,
\begin{equation*}
  \begin{split}
&d(d-1)(d-2)(a\cos^2\theta-d\sin^2\theta)\sin^4\theta+\big[(a^2-a)\cos^2\theta-(d^2-d)\sin^2\theta\big]\\
&~~~~~\times\Big[\frac{a(a-1)(a-2)}{d-a}\cos^2\theta-d(d-2)\sin^4\theta+a(a+2d-2)\sin^2\theta\cos^2\theta\Big]=0.
  \end{split}
\end{equation*}
Notice that the left hand side of the above equation is a polynomial with respect to $\sin^2\theta$,
and the constant term of this polynomial is $\frac{a^2(a-1)^2(a-2)}{d-a}$, then we have
$$\frac{a^2(a-1)^2(a-2)}{d-a}=0,$$
and
$a=0,~1,~2.$
Recall that $a=2d+1$, $d\geq0$ and $d(d-1)(d-2)\neq0$, then
$$a=2,~~~~ d=\frac12.$$
This is also impossible by similar arguments as Case (2.2).
Therefore, Case (3) does't exist.

\textbf{\underline{Case (4) and (5): $b\neq0$, $\delta\leq0$.}} In these two cases,  write
\begin{equation}\label{case4-5}
\left\{
    \begin{array}{ll}
\varphi_1=(C_1f_1+C_2f_2)r^k, \\
\varphi_2=(B_1f_1+B_2f_2)r^k.
    \end{array}
  \right.
\end{equation}
In Case (4),
\begin{equation}\label{case4detail}
\left\{
    \begin{array}{ll}
k=l,\\
f_1=\ln r, ~~~&f_2=1, \\
B_1=\frac{l-a}{b}C_1, ~~~&B_2=\frac{C_1+(l-a)C_2}{b}.
    \end{array}
  \right.
\end{equation}
where $l$ is the unique real root of equation (\ref{char.eqn}).
Since $\varphi_1$ and $\varphi_2$ are linearly independent, then $C_1\neq0$. $\varphi_1, \varphi_2\in C^0$ at $r=0$, then $k=l>0$.

In Case (5),
\begin{equation}\label{case5detail}
\left\{
    \begin{array}{ll}
k=\lambda,\\
f_1=\cos(\mu \ln r), ~~~&f_2=\sin(\mu\ln r), \\
B_1=\frac{(\lambda-a)C_1+\mu C_2}{b}, ~~~&B_2=\frac{(\lambda-a)C_2-\mu C_1}{b}.
    \end{array}
  \right.
\end{equation}
where $\lambda\pm \mu i~(\mu\neq0)$ are the complex roots of equation (\ref{char.eqn}).
Since $\varphi_1$ and $\varphi_2$ are linearly independent, then $C_1^2+C_2^2\neq0$.
$\varphi_1, \varphi_2\in C^0$ at $r=0$, then $k=\lambda>0$.

Substitute (\ref{case4-5}) into (\ref{criti.eqn}) and denote $G=(a+d)v_1v_2$, then we obtain that
\begin{equation}\label{case4-5-a}
  \begin{split}
  &(C_1F_1-B_1F_2)f_1+(C_2F_1-B_2F_2)f_2\\
=r^{k+1}\Big\{&f_1^2\big[(bB_1^2-cC_1^2)G+C_1B_1F_3\big]+f_2^2\big[(bB_2^2-cC_2^2)G+C_2B_2F_3\big]\\
&+f_1f_2\big[2(bB_1B_2-cC_1C_2)G+(C_1B_2+C_2B_1)F_3\big]\Big\}.
  \end{split}
\end{equation}
Since $f_1, f_2, r^{k+1}f_1^2, r^{k+1}f_2^2$ and $r^{k+1}f_1f_2$ are linearly independent,
 the above equations are equivalent to
\begin{equation}\label{case4-5-b}
\left\{
    \begin{array}{ll}
C_1F_1-B_1F_2=0; \\
C_2F_1-B_2F_2=0; \\
(bB_1^2-cC_1^2)G+C_1B_1F_3=0; \\
(bB_2^2-cC_2^2)G+C_2B_2F_3=0; \\
2(bB_1B_2-cC_1C_2)G+(C_1B_2+C_2B_1)F_3=0.
    \end{array}
  \right.
\end{equation}

For Case (4), we substitute (\ref{case4detail}) into the last three equations of (\ref{case4-5-b})
and take into account that $C_1\neq0$, then we obtain that
\begin{numcases}{}
\big[(l-a)^2-bc\big]G+(l-a)F_3=0;    \label{case4.eqn1} \\
\Big\{C_1^2+2(l-a)C_1C_2+\big[(l-a)^2-bc\big]C_2^2\Big\}G+\big[C_1C_2+(l-a)C_2^2\big]F_3=0;  \label{case4.eqn2} \\
\Big\{2(l-a)C_1+2\big[(l-a)^2-bc\big]C_2\Big\}G+\big[C_1+2(l-a)C_2\big]F_3=0.  \label{case4.eqn3}
\end{numcases}
$\frac{(\ref{case4.eqn2})-(\ref{case4.eqn1})\times C^2_2}{C_1}$ gives that
\begin{equation}\label{case4.eqn4}
\big[C_1+2(l-a)C_2\big]G+C_2F_3=0.
\end{equation}
$\frac{(\ref{case4.eqn3})-(\ref{case4.eqn1})\times 2C_2}{C_1}$ gives that
\begin{equation}\label{case4.eqn5}
2(l-a)G+F_3=0.
\end{equation}
$\frac{(\ref{case4.eqn4})-(\ref{case4.eqn5})\times C_2}{C_1}$ gives that
$$G=0,$$
namely
$G=(a+d)v_1v_2=0.$
Notice that $a+d=2l>0$, then
$v_1v_2=0.$
Since $b\neq0$, we apply Lemma \ref{lemma4} again, then we have
$v_1\equiv0,$
which contradicts that $v_1\not\equiv0$.
Therefore Case (4) doesn't exist.

For Case (5), we substitute (\ref{case5detail}) into the last three equations of (\ref{case4-5-b}),
then we obtain that
\begin{numcases}{}
\Big\{\big[(\lambda-a)^2-bc\big]C_1^2+\mu^2C_2^2+2\mu(\lambda-a)C_1C_2\Big\}G
+\big[(\lambda-a)C_1^2+\mu C_1C_2\big]F_3=0;   ~~~~~~~      \label{case5.eqn1} \\
\Big\{\big[(\lambda-a)^2-bc\big]C_2^2+\mu^2C_1^2-2\mu(\lambda-a)C_1C_2\Big\}G
+\big[(\lambda-a)C_2^2-\mu C_1C_2\big]F_3=0;   ~~~~~~~      \label{case5.eqn2} \\
\Big\{2\mu(\lambda-a)\big(C_2^2-C_1^2\big)+2\big[(\lambda-a)^2-\mu^2-bc\big]C_1C_2\Big\}G ~~~~~~\nonumber\\
~~~~~~~~~~~~~~~~~~~~~~~~+\big[\mu \big(C_2^2-C_1^2\big)+2(\lambda-a)C_1C_2\big]F_3=0. ~~~~~~~~   \label{case5.eqn3}
\end{numcases}
We claim that $$G=0.$$
Firstly, if $C_1=0$, then $C_2\neq0$ due to $C_1^2+C_2^2\neq0$, and  (\ref{case5.eqn1}) becomes
$\mu^2C_2^2G=0.$ Since $\mu\neq0$, then $G=0$.\\
Secondly, if $C_2=0$, then $C_1\neq0$ and  (\ref{case5.eqn2}) becomes
$\mu^2C_1^2G=0$ and thus $G=0$.\\
Finally, if $C_1C_2\neq0$, $(\ref{case5.eqn1})\times\frac{C_2}{C_1}+(\ref{case5.eqn2})\times\frac{C_1}{C_2}-(\ref{case5.eqn3})$ gives that
$$\mu^2\frac{\big(C_1^2+C_2^2\big)^2}{C_1C_2}G=0,$$
then
$G=0.$ Thus the claim is proved.\\
Consequently,
$G=(a+d)v_1v_2=0.$
Notice that $a+d=2\lambda>0$, then
$v_1v_2=0,$
which contradicts that $v_1\not\equiv0$.
Therefore Case (5) doesn't exist.

Finally, when $$u=\left(
                              \begin{array}{c}
                                v_1(\theta)\varphi_1(r) \\
                                v_2(\theta)\varphi_2(r) \\
                              \end{array}
                            \right)
,$$
all solutions of equations (\ref{SNS}) are (i), (ii), (iii), (iv) and (v) as shown in Theorem \ref{mainthm1}, (\ref{extrasolu1}) and (\ref{extrasolu2}).

The proof is complete.
\end{proof}

\bigskip

\noindent {\bf Acknowledgments.}
J. Wu is supported by NSFC under grant 11601373.
W. Wang was supported by NSFC under grant 11671067.

\vspace{1cm}

{\small}

\end{document}